\documentclass[11pt,reqno]{amsart}
\usepackage{amsmath,amsthm,amscd,amsfonts,amssymb,color}
\usepackage{cite}
\usepackage[bookmarksnumbered,colorlinks,plainpages]{hyperref}

\newtheorem{theorem}{Theorem}[section]
\newtheorem{lemma}[theorem]{Lemma}

\theoremstyle{definition}
\newtheorem{definition}[theorem]{Definition}
\newtheorem{example}[theorem]{Example}
\theoremstyle{remark}
\newtheorem{remark}[theorem]{Remark}
\numberwithin{equation}{section}
\def\DJ{\leavevmode\setbox0=\hbox{D}\kern0pt\rlap
 {\kern.04em\raise.188\ht0\hbox{-}}D}
\input{tcilatex}

\begin{document}
\title[Zamfirescu mappings in $A$-metric spaces ]{Fixed point results for
Zamfirescu mappings in $A$-metric spaces}

\begin{abstract}
In the present paper, we extend the Zamfirescu results (\cite{16}) to $A$%
-metric spaces. Firstly, we define the notion of Zamfirescu mapping in $A$%
-metric spaces. After, we also obtain a fixed point theorem for such
mappings. The established results carry some well-known results from the
literature (see \cite{b}, \cite{c}, \cite{k}, \cite{16}) to $A$-metric
spaces. Appropriate example is also given.

\smallskip

\end{abstract}

\author{Isa Yildirim}
\thanks{Isa Yildirim, Ataturk University, Department of Mathematics Faculty
of Science, 25240 Erzurum/TURKEY Email address: \textrm{%
isayildirim@atauni.edu.tr}}
\maketitle

\setcounter{page}{1}

\centerline{}

\centerline{}

\section{\protect\bigskip Introduction and Preliminaries}

Zamfirescu's fixed point theorem is one of the most important extensions of
Banach contraction principle. In 1972, Zamfirescu \cite{16} obtained the
following a very interesting fixed point theorem.

\begin{theorem}
\label{t1}Let $(X,d)$ be a complete metric space and $f:X\rightarrow X$ a
mapping for which there exists the real numbers $a,b$ and $c$ satisfying $%
a\in \left( 0,1\right) $;$b,c\in \left( 0,\frac{1}{2}\right) $ such that for
each pair $x,y\in X,$ at least one of the following conditions holds:

$\left( Z_{1}\right) $ $d(fx,fy)\leq ad(x,y);$

$\left( Z_{2}\right) $ $d(fx,fy)\leq b\left[ d(x,fx)+d(y,fy\right] ;$

$\left( Z_{3}\right) $ $d(fx,fy)\leq c\left[ d(x,fy)+d(y,fx\right] .$

Then $f$ has a unique fixed point $p$ and the Picard iteration $\left\{
x_{n}\right\} $ defined by $x_{n+1}=fx_{n}$ converges to $p$ for any
arbitrary but fixed $x_{0}\in X$.
\end{theorem}

An operator $f$ satisfying the contractive conditions $\left( Z_{1}\right)
,\left( Z_{2}\right) $ and $\left( Z_{3}\right) $ in the above theorem is
called Zamfirescu mapping. Zamfirescu's theorem is a generalization of
Banach's, Kannan's and Chatterjea's fixed point theorems (see \cite{b}, \cite%
{c}, \cite{k}). Many researches studied to obtain new classes of contraction
mappings in different metric spaces. Some of them are $D^{\ast }$-metric
space (see \cite{21}), $G$-metric space (see \cite{a5}), $S$-metric space
(see \cite{a4}), $A$-metric space (see \cite{a}), etc. , as a generalization
of the usual metric space. These generalizations helped the development of
fixed point theory.

In the sequel, the letters $%
\mathbb{R}
$, $%
\mathbb{R}
^{+}$, and $%
\mathbb{N}
$ will denote the set of all real numbers, the set of all nonnegative real
numbers, and the set of all positive integers, respectively.

In 2006, Mustafa and Sims \cite{a5} introduced the notion of $G$-metric
space. After, Sedghi et al. \cite{a4} defined the concept of $D^{\ast }$%
-metric and $S$-metric spaces. Also, they proved some fixed point theorems
in such spaces. Every $G$-metric space is a $D^{\ast }$-metric space and
every $D^{\ast }$-metric space is an $S$-metric space. That is, $S$-metric
space is a generalization of the $G$-metric space and the $D^{\ast }$-metric
space.

In 2015, Abbas et al. \cite{a} introduced the concept of an $A$-metric space
as follows:

\begin{definition}
Let $X$ be nonempty set. Suppose a mapping $A:X^{t}\rightarrow 
\mathbb{R}
$ satisfy the following conditions:

$\left( A_{1}\right) $ $A(x_{1},x_{2},...,x_{t-1},x_{t})\geq 0$ $,$

$\left( A_{2}\right) $ $A(x_{1},x_{2},...,x_{t-1},x_{t})=0$ if and only if $%
x_{1}=x_{2}=...=x_{t-1}=x_{t},$

$\left( A_{3}\right) $ $A(x_{1},x_{2},...,x_{t-1},x_{t})\leq
A(x_{1},x_{1},...,(x_{1})_{t-1},y)+A(x_{2},x_{2},...,(x_{2})_{t-1},y)+...+A(x_{t-1},x_{t-1},...,(x_{t-1})_{t-1},y)+A(x_{t},x_{t},...,(x_{t})_{t-1},y) 
$

for any $x_{i},y\in X,$ $(i=1,2,...,t).$ Then, $(X,A)$ is said to be an $A$%
-metric space.
\end{definition}

It is clear that the an $A$-metric space for $t=2$ reduces to ordinary
metric $d$ and an $A$-metric space for $t=3$ reduces to $S$-metric spaces.
So, an $A$-metric space is a generalization of the $G$-metric space, the $%
D^{\ast }$-metric space and the $S$-metric space.

\begin{example}
\cite{a} Let $X=%
\mathbb{R}
$. Define a function $A:X^{t}\rightarrow 
\mathbb{R}
$ by%
\begin{eqnarray*}
A(x_{1},x_{2},...,x_{t-1},x_{t}) &=&\left\vert x_{1}-x_{2}\right\vert
+\left\vert x_{1}-x_{3}\right\vert +...+\left\vert x_{1}-x_{t}\right\vert  \\
&&+\left\vert x_{2}-x_{3}\right\vert +\left\vert x_{2}-x_{4}\right\vert
+...+\left\vert x_{2}-x_{t}\right\vert  \\
&&\vdots  \\
&&+\left\vert x_{t-2}-x_{t-1}\right\vert +\left\vert
x_{t-2}-x_{t}\right\vert +\left\vert x_{t-1}-x_{t}\right\vert  \\
&=&\sum_{i=1}^{t}\sum_{i<j}\left\vert x_{i}-x_{j}\right\vert \text{.}
\end{eqnarray*}

Then $(X,A)$ is an $A$-metric space.
\end{example}

In 2017, Fernandez et al. \cite{fer} introduced the generalized Lipschitz
mapping, Chatterjea's and Kannan's mappings in an $A$-cone metric space over
Banach algebra. Also, they proved some fixed point theorems for the above
mappings in complete $A$-cone metric space $\left( X,A\right) $ over Banach
algebra.

Next,we state the following useful lemmas and definition.

\begin{lemma}
\label{l2} \cite{a} Let $\left( X,A\right) $ be $A$-metric space. Then $%
A\left( x,x,\dots ,x,y\right) =A\left( y,y,\dots ,y,x\right) $ for all $%
x,y\in X$.
\end{lemma}

\begin{lemma}
\label{l3} \cite{a} Let $\left( X,A\right) $ be $A$-metric space. Then for
all for all $x,y\in X$ we have $A\left( x,x,\dots ,x,z\right) \leq \left(
t-1\right) A\left( x,x,\dots ,x,y\right) +A\left( z,z,\dots ,z,y\right) $
and $A\left( x,x,\dots ,x,z\right) \leq \left( t-1\right) A\left( x,x,\dots
,x,y\right) +A\left( y,y,\dots ,y,z\right) $.
\end{lemma}

\begin{definition}
\cite{a} Let $\left( X,A\right) $ be $A$-metric space.

(i) A sequence $\left\{ u_{k}\right\} $ in $X$ is said to converge to a
point $u\in X$ if $A\left( u_{k},u_{k},\dots ,u_{k},u\right) \rightarrow 0$
as $k\rightarrow \infty $.

(ii) A sequence $\left\{ u_{k}\right\} $ in $X$ is called a Cauchy sequence
if $A\left( u_{k},u_{k},\dots ,u_{k},u_{m}\right) \rightarrow 0$ as $%
k,m\rightarrow \infty $.

(iii) The $A$-metric space $\left( X,A\right) $ is said to be complete if
every Cauchy sequence in $X\ $is convergent.
\end{definition}

\section{Main Results}

\noindent In this section, following the ideas of T. Zamfirescu \cite{16} we
first introduce the notion of Zamfirescu mappings in $A$-metric space as
follows:

\begin{definition}
\label{d1} Let $\left( X,A\right) $ be $A$-metric space and $f:X\rightarrow
X $ be a mapping. $f$ is called a $A$-Zamfirescu mapping ($AZ$ mapping), if
and only if, there are real numbers, $0\leq a<1$, $0\leq b,c<\frac{1}{t}$
such that for all $x,y\in X$, at least one of the next conditions is true: 
\begin{equation*}
\left( AZ_{1}\right) A(fx,fx,\dots ,fx,fy)\leq aA(x,x,\dots ,x,y)
\end{equation*}%
\begin{equation*}
\left( AZ_{2}\right) A(fx,fx,\dots ,fx,fy)\leq b\left[ A(fx,fx,\dots
,fx,x)+A(fy,fy,\dots ,fy,y)\right]
\end{equation*}%
\begin{equation*}
\left( AZ_{3}\right) A(fx,fx,\dots ,fx,fy)\leq c\left[ A(fx,fx,\dots
,fx,y)+A(fy,fy,\dots ,fy,x)\right]
\end{equation*}
\end{definition}

\noindent It is clear that if we take $t=2$ in the Definition \ref{d1}, we
obtain the definition of Zamfirescu \cite{16} in ordinary metric space.

\noindent Before giving the our main result in $A$-metric space, we need the
following significant lemma.

\begin{lemma}
\label{l1} Let $\left( X,A\right) $ be $A$-metric space and $f:X\rightarrow
X $ be a mapping. If $f$ is a $AZ$ mapping, then there is $0\leq \delta <1$
such that 
\begin{equation}
A\left( fx,fx,\dots ,fx,fy\right) \leq \delta A\left( x,x,\dots ,x,y\right)
+t\delta A(fx,fx,\dots ,fx,x)  \label{GrindEQ__1_}
\end{equation}%
and 
\begin{equation}
A\left( fx,fx,\dots ,fx,fy\right) \leq \delta A\left( x,x,\dots ,x,y\right)
+t\delta A(fy,fy,\dots ,fy,x)  \label{GrindEQ__2_}
\end{equation}%
for all $x,y\in X$.
\end{lemma}

\begin{proof}
Let's assume that $\left( AZ_{2}\right) $ is hold. From Lemma \ref{l3}, we
have%
\begin{eqnarray*}
A(fx,fx,\dots ,fx,fy) &\leq &b\left[ A(fx,fx,\dots ,fx,x)+A(fy,fy,\dots
,fy,y)\right] \\
&\leq &b\left[ 
\begin{array}{c}
A\left( fx,fx,\dots ,fx,x\right) +\left( t-1\right) A\left( fy,fy,\dots
,fy,fx\right) \\ 
+A(fx,fx,\dots ,fx,y)%
\end{array}%
\right] \\
&\leq &b\left[ A\left( fx,fx,\dots ,fx,x\right) +\left( t-1\right) A\left(
fy,fy,\dots ,fy,fx\right) \right. \\
&&\left. +\left( t-1\right) A\left( fx,fx,\dots ,fx,x\right) +A(x,x,\dots
,x,y)\right]
\end{eqnarray*}%
thus, 
\begin{equation*}
\left( 1-b\left( t-1\right) \right) A\left( fx,fx,\dots ,fx,fy\right) \leq
tbA\left( fx,fx,\dots ,fx,x\right) +bA(x,x,\dots ,x,y).
\end{equation*}%
From the fact that $0\leq b<\frac{1}{t}$, we get 
\begin{equation*}
A\left( fx,fx,\dots ,fx,fy\right) \leq \frac{b}{1-b\left( t-1\right) }%
A\left( x,x,\dots ,x,y\right) +\frac{tb}{1-b\left( t-1\right) }A\left(
fx,fx,\dots ,fx,x\right) .
\end{equation*}%
Assume that $\left( AZ_{3}\right) $ is hold. From Lemmas \ref{l2} and \ref%
{l3}, similarly we get%
\begin{eqnarray*}
A(fx,fx,\dots ,fx,fy) &\leq &c\left[ A(fx,fx,\dots ,fx,y)+A(fy,fy,\dots
,fy,x)\right] \\
&\leq &c\left[ 
\begin{array}{c}
A\left( fx,fx,\dots ,fx,y\right) +\left( t-1\right) A\left( fy,fy,\dots
,fy,fx\right) \\ 
+A\left( fx,fx,\dots ,fx,x\right)%
\end{array}%
\right] \\
&\leq &c\left[ A\left( fx,fx,\dots ,fx,y\right) +\left( t-1\right) A\left(
fy,fy,\dots ,fy,fx\right) \right. \\
&&\left. +\left( t-1\right) A\left( fx,fx,\dots ,fx,y\right) +A(y,y,\dots
,y,x)\right]
\end{eqnarray*}%
thus, 
\begin{equation*}
\left( 1-c\left( t-1\right) \right) A\left( fx,fx,\dots ,fx,fy\right) \leq
tcA\left( fx,fx,\dots ,fx,y\right) +cA(x,x,\dots ,x,y).
\end{equation*}%
From the fact that $0\leq c<\frac{1}{t}$, we get 
\begin{equation*}
A\left( fx,fx,\dots ,fx,fy\right) \leq \frac{c}{1-c\left( t-1\right) }%
A\left( x,x,\dots ,x,y\right) +\frac{tc}{1-c\left( t-1\right) }A\left(
fx,fx,\dots ,fx,y\right) .
\end{equation*}%
Therefore, denoting by 
\begin{equation*}
\delta =max\left\{ a,\frac{b}{1-b\left( t-1\right) },\frac{c}{1-c\left(
t-1\right) }\right\} ,
\end{equation*}%
we have that $0\leq \delta <1$. Thus, the following inequalities hold: 
\begin{equation}
A\left( fx,fx,\dots ,fx,fy\right) \leq \delta A\left( x,x,\dots ,x,y\right)
+t\delta A(fx,fx,\dots ,fx,x)  \label{1}
\end{equation}%
and 
\begin{equation}
A\left( fx,fx,\dots ,fx,fy\right) \leq \delta A\left( x,x,\dots ,x,y\right)
+t\delta A(fy,fy,\dots ,fy,x)  \label{2}
\end{equation}%
for all $x,y\in X$.
\end{proof}

\begin{theorem}
\label{A} \noindent Let $\left( X,A\right) $ be complete $A$-metric space
and $f:X\rightarrow X$ be an$\ AZ$ mapping. Then $f$ has a unique fixed
point and Picard iteration process $\left\{ x_{n}\right\} $ defined by $%
x_{n+1}=fx_{n}$ converges to a fixed point of $f$.
\end{theorem}

\begin{proof}
Let $x_{0}\in X$ be arbitrary and $\left\{ x_{n}\right\} $ be the Picard
iteration as $x_{n+1}=fx_{n}$.

If we take $x=x_{n}$ and $y=x_{n-1}$ at the inequality (\ref{1}), we obtain
that 
\begin{eqnarray*}
A\left( x_{n+1},x_{n+1},\dots ,x_{n+1},x_{n}\right) &=&A\left(
fx_{n},fx_{n},\dots ,fx_{n},fx_{n-1}\right) \\
&\leq &\delta A\left( x_{n},x_{n},\dots ,x_{n},x_{n-1}\right) +t\delta
A\left( x_{n},x_{n},\dots ,x_{n},x_{n}\right) .
\end{eqnarray*}%
From above the inequality, we get 
\begin{eqnarray}
A\left( x_{n+1},x_{n+1},\dots ,x_{n+1},x_{n}\right) &\leq &\delta A\left(
x_{n},x_{n},\dots ,x_{n},x_{n-1}\right)  \label{3} \\
&\leq &{\delta }^{2}A\left( x_{n-1},x_{n-1},\dots ,x_{n-1},x_{n-2}\right) 
\notag \\
&&\vdots  \notag \\
&\leq &{\delta }^{n}A\left( x_{1},x_{1},\dots ,x_{1},x_{0}\right)  \notag \\
&=&{\delta }^{n}A\left( x_{0},x_{0},\dots ,x_{0},x_{1}\right) .  \notag
\end{eqnarray}%
Let $m>n$. Using Lemma \ref{l3} and the above inequality, we get%
\begin{eqnarray*}
A\left( x_{n},x_{n},\dots ,x_{n},x_{m}\right) &\leq &\left( t-1\right)
\delta A\left( x_{n},x_{n},\dots ,x_{n},x_{n+1}\right) +A\left(
x_{n+1},x_{n+1},\dots ,x_{n+1},x_{m}\right) \\
&\leq &\left( t-1\right) \delta A\left( x_{n},x_{n},\dots
,x_{n},x_{n+1}\right) +\left( t-1\right) A\left( x_{n+1},x_{n+1},\dots
,x_{n+1},x_{n+2}\right) \\
&&+A\left( x_{n+2},x_{n+2},\dots ,x_{n+2},x_{m}\right) \\
&\leq &\left( t-1\right) \delta A\left( x_{n},x_{n},\dots
,x_{n},x_{n+1}\right) +\left( t-1\right) A\left( x_{n+1},x_{n+1},\dots
,x_{n+1},x_{n+2}\right) \\
&&+\cdots +\left( t-1\right) A\left( x_{m-2},x_{m-2},\dots
,x_{m-2},x_{m-1}\right) \\
&&+A\left( x_{m-1},x_{m-1},\dots ,x_{m-1},x_{m}\right) \\
&=&\left( t-1\right) \left[ {\delta }^{n}A\left( x_{0},x_{0},\dots
,x_{0},x_{1}\right) +\right. {\delta }^{n+1}A\left( x_{0},x_{0},\dots
,x_{0},x_{1}\right) \\
&&\left. +\cdots +{\delta }^{m-2}A\left( x_{0},x_{0},\dots
,x_{0},x_{1}\right) \right] +{\delta }^{m-1}A\left( x_{0},x_{0},\dots
,x_{0},x_{1}\right) \\
&=&\left( t-1\right) {\delta }^{n}A\left( x_{0},x_{0},\dots
,x_{0},x_{1}\right) \left[ 1+\delta +{\delta }^{2}+\cdots +{\delta }^{m-n-2}%
\right] \\
&&+{\delta }^{m-1}A\left( x_{0},x_{0},\dots ,x_{0},x_{1}\right) \\
&\leq &\left( t-1\right) {\delta }^{n}A\left( x_{0},x_{0},\dots
,x_{0},x_{1}\right) \frac{{\delta }^{m}}{1-\delta }+{\delta }^{m-1}A\left(
x_{0},x_{0},\dots ,x_{0},x_{1}\right) \\
&=&\left[ \left( t-1\right) \frac{{\delta }^{m+n}}{1-\delta }+{\delta }^{m-1}%
\right] A\left( x_{0},x_{0},\dots ,x_{0},x_{1}\right)
\end{eqnarray*}

We know that $0\leq \delta <1$ from Lemma \ref{l1}. Suppose that $A\left(
x_{0},x_{0},\dots ,x_{0},x_{1}\right) >0$. If we take limit as $%
m,n\rightarrow \infty $ in above inequality we get%
\begin{equation*}
{\mathop{\mathrm{lim}}_{n,m\rightarrow \infty }A\left( x_{n},x_{n},\dots
,x_{n},x_{m}\right) =0}\text{.}
\end{equation*}%
Therefore $\left\{ x_{n}\right\} $ is a Cauchy sequence in $X$. Also, assume
that $A\left( x_{0},x_{0},\dots ,x_{0},x_{1}\right) =0$, then $A\left(
x_{n},x_{n},\dots ,x_{n},x_{m}\right) =0$ for all $m>n$ and $\left\{
x_{n}\right\} $ is a Cauchy sequence in $X.$ Since $\left( X,A\right) $ is a
complete metric space, $x_{n}\rightarrow p\in X$ as $n\rightarrow \infty $.

We show that $p$ is a fixed point of $f$. From (\ref{2}), we have 
\begin{eqnarray*}
A\left( fp,fp,\dots ,fp,p\right) &\leq &\left( t-1\right) A\left(
fp,fp,\dots ,fp,fx_{n}\right) +A\left( fx_{n},fx_{n},\dots ,fx_{n},p\right)
\\
&\leq &\left( t-1\right) \left[ \delta A\left( p,p,\dots ,p,x_{n}\right)
+t\delta A\left( p,p,\dots ,p,{fx}_{n}\right) \right] \\
&&+A\left( x_{n+1},x_{n+1},\dots ,x_{n+1},p\right) \\
&=&\left( t-1\right) \left[ \delta A\left( p,p,\dots ,p,x_{n}\right)
+t\delta A\left( p,p,\dots ,p,x_{n+1}\right) \right] \\
&&+A\left( x_{n+1},x_{n+1},\dots ,x_{n+1},p\right) .
\end{eqnarray*}%
If we take limit for $n\rightarrow \infty $ in above inequality, we obtain
that $fp=p$. That is, $p$ is a fixed point of the mapping $f$. Now, we show
that the uniqueness of fixed point of $f$. Assume that $p$ and $q$ are fixed
point of $f$. That is, $fp=p$ and $fq=q$. From (\ref{1}), we have 
\begin{equation*}
A\left( p,p,\dots ,p,q\right) =A\left( fp,fp,\dots ,fp,fq\right) \leq \delta
A\left( p,p,\dots ,p,q\right) +t\delta A\left( fp,fp,\dots ,fp,p\right)
\end{equation*}%
thus,%
\begin{equation*}
A\left( p,p,\dots ,p,q\right) \leq \delta A\left( p,p,\dots ,p,q\right) .
\end{equation*}%
This implies that $A\left( p,p,\dots ,p,q\right) =0\Longrightarrow $ $p=q$
and hence, $f$ has a unique fixed point in $X$.
\end{proof}

\begin{remark}
Putting $t=2$ in Theorem \ref{A}, we obtain the Theorem \ref{t1}. Hence,
Theorem \ref{A} is a generalization of Theorem \ref{t1} of Zamfirescu \cite%
{16} in $A$-metric space.
\end{remark}

\begin{example}
Let $X=%
\mathbb{R}
$. Define a function $A:X^{t}\rightarrow \left[ 0,\infty \right) $ by%
\begin{equation*}
A(x_{1},x_{2},...,x_{t-1},x_{t})=\sum_{i=1}^{t}\sum_{i<j}\left\vert
x_{i}-x_{j}\right\vert 
\end{equation*}

for all $x_{i}\in X$, \ $i=1,2,...,t$. Then $\left( X,A\right) $ is complete 
$A$-metric space.

If we define $f:X\rightarrow X$ by $fx=\frac{2x}{7}$, then $f$ satisfy the
conditions of Theorem \ref{A}. For all $x_{i}\in X$, \ $i=1,2,...,t$,%
\begin{eqnarray*}
A(fx_{1},fx_{2},...,fx_{t-1},fx_{t}) &=&A(\frac{2x_{1}}{7},\frac{2x_{2}}{7}%
,...,\frac{2x_{t-1}}{7},\frac{2x_{t}}{7}) \\
&=&\frac{2}{7}\left\vert x_{1}-x_{2}\right\vert +\frac{2}{7}\left\vert
x_{1}-x_{3}\right\vert +...+\frac{2}{7}\left\vert x_{1}-x_{t}\right\vert  \\
&&+\frac{2}{7}\left\vert x_{2}-x_{3}\right\vert +\frac{2}{7}\left\vert
x_{2}-x_{4}\right\vert +...+\frac{2}{7}\left\vert x_{2}-x_{t}\right\vert  \\
&&\vdots  \\
&&+\frac{2}{7}\left\vert x_{t-2}-x_{t-1}\right\vert +\frac{2}{7}\left\vert
x_{t-2}-x_{t}\right\vert +\frac{2}{7}\left\vert x_{t-1}-x_{t}\right\vert  \\
&=&\frac{2}{7}\sum_{i=1}^{t}\sum_{i<j}\left\vert x_{i}-x_{j}\right\vert  \\
&=&\frac{2}{7}A(x_{1},x_{2},...,x_{t-1},x_{t})
\end{eqnarray*}%
where $\frac{2}{7}\in \left[ 0,1\right) $. This implies that $f$ is a $AZ$
mapping. And $x=0$ is the unique fixed point of $f$ in $X$ as asserted by
Theorem \ref{A}.
\end{example}

\end{document}